\documentclass[12pt]{amsart}

\usepackage[english]{babel}

\usepackage{times,amsfonts,amsmath,amssymb,dsfont} 
\usepackage{pdfsync}
\usepackage[notref,notcite]{showkeys}
\usepackage{graphicx} 
\usepackage{mathabx}

\usepackage{cases}

\usepackage{soul}
\setstcolor{red}

\usepackage{color}

\definecolor{gr}{rgb}   {0.,   0.69,   0.23 }
\definecolor{bl}{rgb}   {0.,   0.5,   1. }
\definecolor{mg}{rgb}   {0.85,  0.,    0.85}
\definecolor{or}{rgb}   {0.9,  0.5,   0.}

\definecolor{webred}{rgb}{0.75,0,0}
\definecolor{webgreen}{rgb}{0,0.75,0}
\usepackage[citecolor=webgreen,colorlinks=true,linkcolor=webred]{hyperref}

\newtheorem{theorem}{Theorem}[section]

\theoremstyle{definition}

\theoremstyle{remark}
\newtheorem{remark}[theorem]{Remark}

\newcommand{\ba}{\begin{array}}
\newcommand{\ea}{\end{array}}

\newcommand{\Bk}{\color{black}}

\newcommand{\N}{\mathbb{N}}
\newcommand{\R}{\mathbb{R}}

\newcommand{\rd}{{\mathrm d}}

\newcommand{\cF}{\mathcal{F}}

\newcommand{\cT}{\mathcal{T}}

\newcommand{\gh}{\mathfrak{h}}

\newcommand{\bel}{\begin{equation} \label}
\newcommand{\ee}{\end{equation}}

\newcommand\dom{\operatorname{Dom}}

\begin{document}\setul{2.5ex}{.25ex}
\title[]{Band functions of Iwatsuka models : power-like and flat magnetic fields}
\author[P.\ Miranda]{Pablo Miranda}
\address{Departamento de Matem\'atica y Ciencia de la Computaci\'on, Universidad de Santiago de Chile, Las Sophoras 173. Santiago, Chile.}\email{pablo.miranda.r@usach.cl}
\author[N.\ Popoff]{Nicolas Popoff}
\address{Universit\'e de Bordeaux, IMB, UMR 5251, 33405 TALENCE cedex, France}
\email{Nicolas.Popoff@math.u-bordeaux1.fr}
\email{}

\maketitle

\begin{abstract}
In this note we consider the Iwatsuka model with a postive increasing magnetic field having finite limits. The associated magnetic Laplacian is fibred through partial Fourier transform, and, for large frequencies, the band functions tend to the Landau levels, which are thresholds in the spectrum. The asymptotics of the band functions is already known when the magnetic field converge polynomially to its limits. We complete this analysis by giving the asymptotics for a regular magnetic field which is constant at infinity, showing that the band functions converge now exponentially fast toward the thresholds. As an application, we give a control on the current of quantum states localized in energy near a threshold.  

\end{abstract}

\section{The Iwatsuka model}

In this article we review and complete some results about the band function of the Iwtasuka model with an increasing positive magnetic field having finite limits.
Assume that the magnetic field $b:\R^{2}\to (0,+\infty)$ depends only of one variable in the sense that $b(x,y)=b(x)$. We assume moreover that $b$ is $C^{0}$, increasing, and have finite limits $b^{\pm}$ as $x\to\pm\infty$, with $0<b_{-}<b_{+}$.  The model is gauge invariant, and we choose 
the magnetic potential 
$$A(x,y):=(0,a(x)); \quad \textrm{with}  \quad a(x):=\int_{0}^{x}b(t) \rd t.$$\Bk
The magnetic Laplacian is then defined by
$$H_{0}:=(-i\nabla-A)^2=-\partial_{x}^2+(-i\partial_{y}-a(x))^2$$
acting in $L^2(\R^2)$. 

Historically, this operator was introduced in order to provide an example of a magnetic Laplacian with absolutely continous spectrum see \cite{Iwa85}. Then, this has been proved under various conditions on $b$, see \cite{Iwa85, ManPur97,ExKov00, Tus16}, but the fact that this is true as long as $b$ is non-constant is still open, since \cite{CyFrKiSi87}.


With the years, this model has been widely studied, as a source of interesting questions linked to transport phenomena in the translationnally invariance direction. In this note, we describe and complete results on the asymptotics of the band functions, which is  an important step when describing spectral properties at energy near the thresholds.


A key tool for operators having a translation invariance is fibration through partial Fourier transform.
In our case, denote by $\cF_{y}$ the partial Fourier transform in the $y$ variable. Then there holds 
$$\cF_{y}H_{0}\cF_{y}^{\star} =\int^{\oplus} \gh(k)\rd k$$
where $\gh(k)$ is the unidimensional Sturm-Liouville operator defined by
\bel{D:hk}
\gh(k):=-\partial_{x}^2+(a(x)-k)^2
\ee
acting on $L^{2}( \R \Bk)$. It is positive, self-adjoint with compact resolvent.  We  denote by $\{E_{n}(k), n\geq1\}$ the increasing sequence of  its  eigenvalues. They are simple, see \cite[Lemma 2.3]{Iwa85},  therefore the functions $E_n(\cdot)$ are analytic with respect to $k$ on $\R$. They are called the band functions (or dispersion curves) of $H_{0}$.


With the hypotheses assumed here, the band functions $k\mapsto E_{n}(k)$ are increasing and converges to $\Lambda_{n}b^{\pm}$ as $k\to\pm\infty$, where $\Lambda_n:=2n-1$. In this case, the 
 spectrum is obviously purely absolutely continuous. 

The values  $\Lambda_{n}b^{\pm}$ \Bk  are thresholds in the spectrum of $H_{0}$. The nature of these thresholds,  and more refined properties \Bk of the operator (and its perturbations), are deeply link to the behavior of the band functions at these limits. 
 
The trajectory of a classical particle submitted to this kind of magnetic fields is quite easy to picture.  Generically,  the  particles exhibit a drift in the invariance direction $y$. Because the magnetic field varies slowly when $|x|$ is large, if the particle is located initially in such a zone, the drift will be weak, and the trajectory of the particle will be close to circles. For a spinless quantum particle of a given energy, the evolution is linked to the band functions crossing this energy, the velocity in the $y$ direction being related to the derivative of the band functions. If the energy is far from thresholds, the particle is usually called an {\it edge state}, because it will show some propagation, as it is the case for models involved models in Quantum Hall Effect, where edges induce transport. We refer to \cite{ManPur97,PeetRej00,HisSoc14} for study of this case. On the contrary, if the energy is closed to a threshold, the particle will bear a {\it bulk component}, that is a component whose velocity is small. Quantitative estimates on the velocity requires asymptotics of the band functions near thresholds. Information on the eigenfunctions of the fiber operator involved (de)localization of the particle, through an analysis in phase space (see \cite{HisPofSoc14,MirPof18}, and also \cite{dBiePu99} for a rough analysis on a similar model).

\section{Asymptotics of band functions}
In \cite{MirPof18}, it is assumed that the magnetic field converges to its limit like negative power of $x$, the model case being 
$$\exists x_{0}\in \R, \forall x \geq x_{0}, \quad b(x)=b_{+}-\langle x \rangle^{\alpha}, \quad \alpha>0.$$
Our hypotheses bear on the behavior of the magnetic field at $+\infty$ and the asymptotics of the band function when $k\to+\infty$. Symmetric hypotheses and results are of course valid in the other direction.\Bk

Under this condition, the behavior of the band functions is provided in \cite[Theorem 2.2 and Corollary 2.4]{MirPof18}: 
\bel{H:a}
E_{n}(k)=b_{+}\Lambda_{n}-\frac{\Lambda_{n}b_{+}^{M}}{k^{M}}+O\big(\frac{1}{k^{M+2}}\big).
\ee
Here we will consider another physically relevant class of magnetic field, those which are {\it equal} to their limit for large $x$:
\bel{H:b}
\exists x_{\infty}\in \R, \forall x \geq x_{\infty}, \quad b(x)=b_{+}.
\ee

Up to picking the smallest real satisfying the above relation, we also assume that $b(x)<b_{+}$ for $x<x_{\infty}$. 
The case of a piecewise constant magnetic field was treated in \cite{PeetRej00,HisPerPofRay16}, but in that case, the use of special functions allows precise computations of the asymptotics, and are not available in a general context. Moreover, in this article we are interested in more regular magnetic fields.\Bk

We say that the contact at $x_{\infty}$ is of order $p\in \N\cup\{+\infty\}$ when  $b$ is $C^{p}((-\infty,x_\infty)$, $$\lim_{\substack{x\to x_\infty \\ x<x_{\infty}}}b^{(p)}(x):=b^{(p)}(x_{\infty}^{-})\neq0$$ \Bk and $b^{(j)}(x_{\infty})=0$ for all $j=1,\ldots,p-1$. 

The potential in \eqref{D:hk} vanishes at a unique point $a^{-1}(k)=:x_{k}$, $a^{-1}$ being the  inverse function of $a$. The proof of the asymptotics relies on the construction of quasi-modes in the spirit of the harmonic approximation \cite{DiSj99}. Indeed, after the change of variable $x=b_{+}^{-1/2}t+x_{k}$, the operator $\gh(k)$ is transformed in $b_{+}\widetilde{\gh}(k)$, whith
\bel
{D:tildegh}
\widetilde{\gh}(k)=-\partial_{t}^2+W(t,k).
\ee
In both cases, \eqref{H:a} and $\eqref{H:b}$, $W$ has a unique minimum at 0 which is non degenerate, in the sense that $\partial_{t}^2W(0,k)=2$.

Writing $W(t,k)=t^2+d_{k}(t)$, then, in case \eqref{H:a}, $d_{k}$ is not zero near $t=0$,  and $d_{k}(t)=\alpha(k)O(t^3)$, where $\alpha(k)\to0$ as $k\to+\infty$. Therefore perturbation theory provides the asymptotics of the eigenvalues of $\gh(k)$ as $k\to+\infty$. 
But in case \eqref{H:b}, $W(t,k)=t^2$ in a neighborhood of 0 and it is not clear on which quantity depends the asymptotics of the band function, and which is the convergence rate.  We define $a_{\infty}:=a(x_{\infty})$, and
 $$t_{k}:=\sqrt{b_{+}}(x_{\infty}-x_{k})=\frac{a_{\infty}-k}{b_{+}^{1/2}}.$$
Notice that for $k$ large enough, $t_{k}<0$, and $W(t,k)=t^2$ for $t\in (t_{k},+\infty)$.
\Bk
The asymptotics of the band functions is now very different from \eqref{H:a}:  as one expect, the band functions converge faster to its limit, more precisely the quantity $E_{n}(k)-b_{+}\Lambda_{n}$ is now exponentially small as $k\to+\infty$, and the first order term depends only on the contact point, as follows: 

\begin{theorem}
\label{asymptotic_flat_field}
Assume that the contact at $x_{\infty}$ is of order $p \geq1$. Then, as $k\to+\infty$:
\bel{E:asEn}
E_{n}(k)=\Lambda_{n}+C(n,p,b_{+})b^{(p)}(x_{\infty}^{-})k^{2n-p-3}e^{-t_{k}^2}+o(k^{2n-p-3}e^{-t_{k}^2}),
\ee
where $C(n,p,b_{+})=(-1)^{p}\frac{2^{n-p-2}}{\sqrt{\pi}(n-1)!b_{+}^{n-\frac{3}{2}}}$. 

If the contact is of order $\infty$, then 
\bel{A:infinitecontact}
e^{t_{k}^2}(E_{n}(k)-\Lambda_{n})=o(k^{-\infty}).
\ee
\end{theorem}
\begin{proof}
Assume \eqref{H:b} and note that  
\bel{E:xk}
x_{k}=x_{\infty}+\frac{k-a_{\infty}}{b_{+}},
\ee
moreover, $b$ is constant on $(x_{\infty},x_{k})$. 
Recall that
$d_{k}(t)=W(k,t)-t^2$, which  
for all $t\in \R$ can be written as:
\bel{E:dk}
d_{k}(t)=\frac{1}{b_{+}}\int_{x_{k}+\frac{t}{\sqrt{b_{+}}}}^{x_{k}}(b(s)+b_{+})\rd s\int_{x_{k}+\frac{t}{\sqrt{b_{+}}}}^{x_{k}}(b(s)-b_{+})\rd s.
\ee
Note that $|d_{k}(t)| \leq Ct^2$ and that $d_{k}\to0$ point-wise as $k\to+\infty$. Therefore, we will approximate $\widetilde{\gh}(k)$ by the harmonic oscillator $-\partial_{t}^2+t^2$ as $k\to+\infty$. The same computations as in \cite[Section 2.1]{MirPof18}, based of construction of quasi-modes for $\widetilde{\gh}(k)$ from Hermite's functions, and estimations of remainders, applies. Denote by $\widetilde{E_{n}}(k)$ the $n$-th eigenvalue of $\widetilde{\gh}(k)$, then (we skip the details for brevity), as $k\to+\infty$:
\bel{A:Entilde}
\widetilde{E_{n}}(k)=\Lambda_{n}+\mu_{n}(k)+o(\mu_{n}(k)),
\ee
where 
\bel{E:munini}
\mu_{n}(k)=\int_{\R}\Psi_{n}(t)^2d_{k}(t) \rd t =\int_{-\infty}^{t_{k}}P_{n}^2(t)e^{-t^2} d_{k}(t) \rd t,
\ee
$\Psi_{n}$ being the $n$-th normalized Hermite's function, starting from $n=1$, and $P_{n}$ the associated $n$-th Hermite's polynomial. Since $P_{n}$ is of degre $n-1$ with leading coefficient $\gamma_{n}:=(2^{n-1}/((n-1)!\pi^{1/2}))^{1/2}$, we get
\bel{E:muniniequiv}
\mu_{n}(k)\underset{k\to+\infty}{\sim} \gamma_{n}^2\int_{-\infty}^{t_{k}}t^{2n-2}e^{-t^2}d_{k}(t)\rd t.
\ee
Relation \eqref{E:dk} and the definition of $t_{k}$ provides $d_{k}^{(j)}(t_{k})=0$ for all $j=1,\ldots,p$, moreover $|d_{k}^{(p+2)}(t)|\leq C_{p}|t|$ with $C_{p}>0$. Therefore, using $(p+1)$ integrations by parts, we get: 
$$\int_{-\infty}^{t_{k}}t^{2n-2}e^{-t^2}d_{k}(t)\rd t\underset{k\to+\infty}{\sim} -t_{k}^{2n-2}\frac{e^{-t_{k}^2}}{(2t_{k})^{p+2}}d^{(p+1)}(t_{k}).$$
Now, from \eqref{E:xk}:
\begin{equation}
\label{calculdp}
d_{k}^{(p+1)}(t_{k})=-2(k-a_{\infty})b_{+}^{-\frac{p+3}{2}}b^{(p)}(x_{\infty}^{-}).
\end{equation}
Therefore, by  the definition of $t_{k}$:

$$\mu_{n}(k)\underset{k\to+\infty}{\sim} \frac{2^{n-p-2}}{\sqrt{\pi}(n-1)!b_{+}^{n-\frac{1}{2}}}b^{(p)}(x_{\infty}^{-})k^{2n-p-3}e^{-t_{k}^2}.$$
Using \eqref{A:Entilde} and $E_{n}(k)=b_{+}\widetilde{E_{n}}(k)$, we get \eqref{E:asEn}. 

When the contact is of infinite order, we have $d^{(j)}(t_{k})=0$ for all $j\in \N$. Then, \eqref{A:infinitecontact} follows easily from \eqref{E:munini}.
\end{proof}
\begin{remark}
In case of an infinite contact point, the proof shows that the asymptotics of the band function is still given by $E_{n}(k)=b_{+}\Lambda_{n}+b_{+}\mu_{n}(k)+o(\mu_{n}(k))$, where $\mu_{n}$ is given in \eqref{E:munini}. For $k$ large enough, $\mu_{n}(k)<0$, but there is no natural expansion for this quantity without additional hypotheses. 
\end{remark}


\section{Consequences near the thresholds}
\subsection{Bulk states}
The current operator is defined as the commutator $J_{y}:=-i[H_{0},y]$, on $\dom(H_{0})$. The evolution through the unitary group defined by $H_{0}$ of this self-adjont operator is the velocity in the $y$ direction, indeed, defining $y(t):=e^{-itH_{0}}ye^{itH_{0}}$, the evolution of the position along the $y$ direction, there holds
$$\frac{\rd y(t)}{\rd t}=e^{-iH_{0}}J_{y}e^{itH_{0}}.$$
In \cite{GeNi98}, the authors use this commutator to establish a Mourre estimate for general analytically fibered Hamiltonian for any energies except for a discrete set, called thresholds. Their techniques  relies on the stratification of the projection from the Bloch Variety into the spectrum, according to the algebraic multiplicity of values in the spectrum, the thresholds being the set of dimension 0 in the spectrum associated with this stratification, see \cite[Definition 3.2]{GeNi98}. In some sense they provide a general estimate from below from the current: it is bounded from below at energies away from thresholds. 

Iwatsuka models, among others magnetic models, have the property that their fibers operators are 1d, and the algbraic machinery described above can be avoided by controlling the commutator in a more direct way. The key tool to do this is the Feymann-Hellman formula. Given a function $\varphi\in\dom(H_{0})$ (see \cite[Section 3]{MirPof18} for a precise definition), there holds:
$$\langle J_{y}\pi_{n}\varphi,\pi_{n}\varphi \rangle=\int_{k\in \R}|\pi_{n}\varphi(k)|^2 \lambda_{n}'(k) \rd k,$$
where $\pi_{n}$ is the projection along the $n$-th harmonic, see \cite[Section 5]{ManPur97} and \cite{HisSoc14}.

Therefore, estimates on the derivative on the band function turn into control on the current operator, more precisely, if we assume in addition that $\varphi$ is localized in energy in an interval $I$, then 
\bel{estimecourant} \inf_{E_{n}^{-1}(I)}E_{n}' \leq \frac{\langle J_{y}\pi_{n}\varphi,\pi_{n}\varphi\rangle}{\|\pi_{n}\varphi\|^2} \leq \sup_{E_{n}^{-1}(I)}E_{n}'.
\ee
If thresholds are defined as set of energies for which the derivative of the band function can be small, then it becomes obvious that the current is bounded from below, away from thresholds.
 In case where the band functions are proper, such energies correspond to critical point of band functions. Note that being proper is also an hypothesis from \cite{GeNi98}. But in our case, the band functions tends to finite limit, giving rise to a different kind of thresholds $\cT:=\{b_{\pm}\Lambda_{n}, n\geq1\}$. For a particle whose energy interval contains a threshold, no bound from below is available for the current. A more quantitative approach is given in \cite{HisPofSoc14}: we consider an energy interval $I$ at a small distance from the set of threshold $\cT$. Then \eqref{estimecourant} shows that a precise asymptotics of $E_{n}$ and its derivative provide a good control on the current of states localized in energy in $I$. Following this strategy, it is proved in \cite{MirPof18} that in case \eqref{H:a}, when $I=(\lambda_{n}-\delta_{2},\Lambda_{n}-\delta_{1})$, then, as $\delta_{i}\to0$:
 $$\delta_{1}^{1+\frac{1}{M}} \lesssim \frac{\langle J_{y}\pi_{n}\varphi,\pi_{n}\varphi\rangle}{\|\pi_{n}\varphi\|^2} \lesssim \delta_{2}^{1+\frac{1}{M}}.$$
In the case where the magnetic field satisfies \eqref{H:b}, then the strategy is similar. First, for a fixed $\delta>0$ small enough, the equation $\lambda_{n}-\delta=E_{n}(k)$ has a unique solution $k(\delta)$, which satisfies, $k(\delta)=\sqrt{|b_{+}\log\delta|}+o(1)$ as $\delta\to0$.

Next one needs to show that the asymptotics  of $E_{n}'(k)$ can be derived from \eqref{E:asEn}. Since the magnetic field does not satisfies any analyticity hypothesis, no method based on special functions can be used (as it was the case in \cite{GeSe97,HisPerPofRay16}). A direct method is to use intergal formula for $E_{n}'(k)$ and a refined asymptotics of the eigenfunction, as done in \cite[Section 2.3]{MirPof18}. We do not explained this in full details here. As a consequence, we get 
$$E_{n}'(k(\delta)) \underset{\delta\to0}{\sim}\delta \sqrt{|\log\delta|}$$
and therefore the estimates on the current follows: 
 $$\delta_{1} \sqrt{|\log\delta_{1}|} \lesssim \frac{\langle J_{y}\pi_{n}\varphi,\pi_{n}\varphi\rangle}{\|\pi_{n}\varphi\|^2} \lesssim \delta_{2} \sqrt{|\log\delta_{2}|},$$
showing the difference with case \eqref{H:a}. Estimations on the localizations on states localized in this interval are posible, as done in \cite[Section 3.2]{MirPof18}.

\subsection{Perturbations}
Here we describe other possible applications of the asymptotics of band functions, without entering the details. 

Giving a measurable sign-definite potential $V:\R^{2}\to \R$ going to 0 at infinity, a physically relevant question concerns the effect of this perturbation on the system. Some of the classical trajectories may become bounded, corresponding to {\it trapped modes}. These one correspond in the quantum system to eigenvalues of the operator $H_{0}+V$. If there are gaps in the  essential \Bk spectrum of $H_{0}$, then under suitable assumptions on $V$, these gaps remains the same. But discrete \Bk eigenvalues can appear inside these gaps. Finiteness (or asymptotics) of these eigenvalues is an important topics which has received a lot of attention. For example, the general case of thresholds corresponding to critical points has been described in \cite{Rai92}, but in our case, these technics does not apply since the thresholds are limit of band functions. These cases are treated under various hypotheses on the potential in \cite{shi03,shi05,Mir16}. 
A delicate extension of these question concerns the behavior of the spectral shift function for the pair $(H_{0}+V,H_{0})$. This function can be used to describe the counting function of eigenvalues outside the essential spectrum, and its singularity at thresholds is a natural question. This problematic has been treated in \cite[Section 4]{MirPof18} for the Iwatsuka model, under condition \eqref{H:a}.

 In case \eqref{H:b}, the exponential convergence of the band functions toward their limits is closer to half-planes model with constant magnetic field. Therefore the precise behavior of the eigenvalues counting function, and of the SSF, can be obtained by adapting the methods from \cite{BruMirRai13,BruMir16}, using the asymptotics \eqref{E:asEn} for $E_{n}$ (and its derivatives). 



\bibliographystyle{plain}
\bibliography{bibliopof}
\end{document}